\documentclass[a4paper, reqno, 11pt]{amsart}
\usepackage{color}
\makeatletter

\@addtoreset{equation}{section}
\makeatother
\usepackage{setspace}
\usepackage{xcolor}
\usepackage{here}
\usepackage{graphicx}
\usepackage{float}
\usepackage[hidelinks]{hyperref}
\usepackage{tikz}
\usepackage[T1]{fontenc}
\usepackage{amsmath}
\usepackage{amssymb}
\usepackage{latexsym}
\usepackage{amsthm}
\usepackage{hyperref}
\usepackage{mathtools}
\usepackage{mathrsfs}
\usepackage{caption}
\everymath{\displaystyle}

\newtheorem{Thm}{Theorem}[section]

\theoremstyle{definition}

\newtheorem{Rem}[Thm]{Remark}
\newtheorem{Exa}[Thm]{Example}

\begin{document}

\begin{abstract}
We study graph-directed conjugate functional equations on the unit interval indexed by the complete digraph with self-loops on two vertices. 
We focus on the singularity and regularity of the solutions for compatible systems of weak contractions. 
First, we show that both solutions are singular in the affine case unless the two systems coincide; second, we obtain a dichotomy between singularity and smoothness for a class of linear fractional systems; and finally, we give a sufficient condition for singularity in a non-linear setting.
\end{abstract}

\title[]{Singularity for graph-directed conjugate equations indexed by a two-vertex digraph}
\author{Kazuki Okamura}
\date{\today}
\address{Department of Mathematics, Faculty of Science, Shizuoka University, 836, Ohya, Suruga-ku, Shizuoka, 422-8529, JAPAN.}
\email{okamura.kazuki@shizuoka.ac.jp}
\maketitle

\section{Introduction}\label{sec:intro}

The de Rham functional equation belongs to an important class of iterative functional equations and has been studied by many authors (\cite{Barany2018, Berg2000, BuescuSerpa2021, Girgensohn2006, Hata1985K, Hata1985, Kawaharada2025, Kawamura2002, Serpa2015, Serpa2015N, Serpa2017, Zdun2001}).   
It can also be regarded as a special case of conjugate equations \cite{BuescuSerpa2021, Okamura2019}. 
More recently, graph-directed versions of such equations have been considered in \cite{MuramotoSekiguchi2020}, and \cite[Section 5]{Okamura2025} extended the framework on the unit interval to a graph-directed setting.

In this paper, we deal with the simplest nontrivial graph-directed case, specifically, the complete digraph with self-loops on two vertices. 
Let $\{f_{i,j}\}_{i,j \in \{0,1\}}$ and $\{g_{i,j}\}_{i,j \in \{0,1\}}$ be transformations on $[0,1]$ satisfying a certain compatibility condition. 
We consider the pair of solutions $(\varphi_0, \varphi_1)$ of 
\begin{equation}\label{eq:conjugate-1st} 
g_{i,j} \circ \varphi_j = \varphi_{i} \circ f_{i,j}, \ i, j \in \{0,1\}. 
\end{equation}
See Section \ref{sec:frame} for the precise formulation. 
\cite[Section 5]{Okamura2025} establishes the existence and uniqueness of the pair of increasing and continuous solutions $(\varphi_0, \varphi_1)$. 

In this paper, we consider the singularity and regularity of  $(\varphi_0, \varphi_1)$. 
It is a central theme in this area whether solutions are singular or smooth. 
Singularity for de Rham-type functional equations and, more generally, de Rham curves has been investigated. 
However, we are not aware of previous singularity results for graph-directed conjugate equations in the present framework. 

The main results of this paper are as follows.
First, we prove that in the affine case both solutions $\varphi_0$ and $\varphi_1$ are singular unless the two affine systems $\{f_{i,j}\}_{i,j}$ and $\{g_{i,j}\}_{i,j}$ coincide, in which case both solutions are the identity. 
Second,  
we obtain a dichotomy between singularity and smoothness for a family with dyadic source maps $\{f_{i,j}\}_{i,j}$ and linear fractional target maps $\{g_{i,j}\}_{i,j}$, and give explicit expressions for the solutions in the regular case. 
Third, we give a sufficient condition for singularity for a family with dyadic source maps $\{f_{i,j}\}_{i,j}$ and non-linear target maps $\{g_{i,j}\}_{i,j}$ in terms of the product of the Lipschitz constants.

This paper is organized as follows. 
In Section  \ref{sec:frame}, we recall the framework from \cite[Section 5]{Okamura2025}. 
Section \ref{sec:affine} considers the affine case. 
Section \ref{sec:LF} studies the linear fractional case. 
Section  \ref{sec:nl} gives a singularity criterion in a non-linear setting.
Finally, in Section \ref{sec:exa}, we give some examples.

\section{Framework}\label{sec:frame}

We follow the setting of \cite[Section 5]{Okamura2025}. 
However, we use $\{0,1\}$ instead of $\{1,2\}$ as the index set for the maps. 

We adopt the Matkowski--Rus definition of weak contractions \cite{Matkowski1975,Rus2001}.  
We say that a map $\phi : [0,\infty) \to [0,\infty)$ is a {\it comparison function} if 
$\phi$ is increasing and 
$\displaystyle \lim_{n \to \infty} \phi^n (t) = 0$ for every $t > 0$. 
Let $\phi$ be a comparison function. 
We say that a map $h : [0,1] \to [0,1]$ is a {\it $\phi$-contraction} if 
$|h(x) - h(y)| \le \phi(|x-y|)$ for every $x, y \in [0,1]$. 
Since $\phi(t) \le t$ for every $t \ge 0$, every $\phi$-contraction is continuous and its Lipschitz norm is less than or equal to $1$. 

We say that a doubly-indexed family of functions  $h_{i,j}, i, j \in \{0,1\}$, on $[0,1]$ is a {\it compatible system} with a comparison function $\phi$ 
if each $h_{i,j}$ is strictly increasing, and a $\phi$-contraction, and furthermore a compatibility condition holds, specifically, 
\begin{equation}\label{eq:compatibility-def} 
0 = h_{i,0} (0) < h_{i,0}(1) = h_{i,1}(0) < h_{i,1}(1) = 1, \ i \in \{0,1\}. 
\end{equation} 

Let 
\[ D^h_{i, n}  \coloneqq \left\{  h_{i,i_1} \circ h_{i_1,i_2} \circ \cdots \circ h_{i_{n-1},i_n} (x_n) \middle| i_1, \dots, i_n \in \{0,1\}, x_n \in \{0,1\}  \right\} \]
and $\displaystyle D^h_{i} \coloneqq \bigcup_{n \ge 1} D^h_{i, n}$. 
Then $(D^h_{i, n})_n$ is increasing with respect to $n$. 

Let 
\[ I^h_{i} (i_1, \dots, i_n)  \coloneqq \left[ h_{i,i_1} \circ h_{i_1,i_2} \circ \cdots \circ h_{i_{n-1},i_n} (0), h_{i,i_1} \circ h_{i_1,i_2} \circ \cdots \circ h_{i_{n-1},i_n} (1) \right]. \]
By the compatibility condition \eqref{eq:compatibility-def}, 
\begin{equation}\label{eq:interval-split} 
I^h_{i} (i_1, \dots, i_n) = I^h_{i} (i_1, \dots, i_n,0) \cup I^h_{i} (i_1, \dots, i_n,1). 
\end{equation}

For every $x \in [0,1]$, 
there exists a sequence $(i_n)_n$ in $\{0,1\}$ such that 
$x \in  I^h_{i} (i_1, \dots, i_n)$ for each $n$, and, if $x \in [0,1] \setminus  D^h_{i}$, then such a sequence is unique. 
We can choose $i_n$ inductively. 
Indeed, by \eqref{eq:interval-split}, 
if $x \in  I^h_{i} (i_1, \dots, i_n)$, then $x \in  I^h_{i} (i_1, \dots, i_n, 0)$ or $x \in  I^h_{i} (i_1, \dots, i_n, 1)$.
If additionally $x \in [0,1] \setminus  D^h_{i}$, then exactly one of $x \in  I^h_{i} (i_1, \dots, i_n, 0)$ and $x \in  I^h_{i} (i_1, \dots, i_n, 1)$ holds.

We denote the length of an interval $I$ by $|I|$. 
Let 
\[ \Delta^{h}_{i, n} \coloneqq \max_{i_1, \dots, i_n \in \{0,1\}}  \left| I^h_{i} (i_1, \dots, i_n) \right|, \ n \ge 1, \  i \in \{0,1\}.  \]
Since $\left| I^h_{i} (i_1, \dots, i_n) \right| \le \phi^n (1)$ and $\phi$ is a comparison function, 
\begin{equation}\label{eq:interval-shrink} 
\lim_{n \to \infty} \Delta^{h}_{i, n} = 0. 
\end{equation}
Hence, $D^h_{i}$ is dense in $[0,1]$. 

Let $\{f_{i,j}\}_{i,j}$ and $\{g_{i,j}\}_{i,j}$ be two compatible systems on $[0,1]$ with a common comparison function $\phi$ under the Euclidean distance. 
For each $i \in \{0,1\}$, 
we define functions $\varphi_i$, on $D^f_{i}$ such that for $i_1, \dots, i_n \in \{0,1\}, x_n \in \{0,1\}$, 
\[ \varphi_i \left( f_{i, i_1} \circ f_{i_1, i_2} \circ \cdots \circ f_{i_{n-1}, i_n} (x_n) \right) = g_{i, i_1} \circ g_{i_1, i_2} \circ \cdots \circ g_{i_{n-1}, i_n} (x_n). \]
This is well-defined by arguments in \cite[Section 5]{Okamura2025}. 

Now we define $\varphi_i (x)$ for $x \in [0,1] \setminus  D^{f}_{i}$. 
Take a unique sequence $(i_n)_n$ in $\{0,1\}$ such that $x \in I^f_{i} (i_1, \dots, i_n)$ for every $n \ge 1$. 
Then, 
$\displaystyle \bigcap_{n \ge 1} I^g_{i} (i_1, \dots, i_n)$ is a compact subset of $[0,1]$ and 
by \eqref{eq:interval-shrink}, 
it is a singleton.
Let $y \in [0,1]$ such that $\displaystyle \{y\} = \bigcap_{n \ge 1} I^g_{i} (i_1, \dots, i_n)$. 
Then we define $\varphi_i (x) \coloneqq y$. 
Thus a map $\varphi_i : [0,1] \to [0,1]$ is defined. 

By \cite[Proposition 5.1]{Okamura2025}, 
each $\varphi_i$ is increasing and continuous on $[0,1]$.
The pair of maps $(\varphi_0, \varphi_1)$ satisfies the following system of functional equations: 
\begin{equation}\label{eq:gd-dR-fe-def-alt} 
\begin{cases} g_{0,0}(\varphi_{0}(x)) = \varphi_0 (f_{0,0}(x))  \\ g_{0,1}(\varphi_{1}(x)) = \varphi_0 (f_{0,1}(x))  \\ g_{1,0}(\varphi_{0}(x)) = \varphi_1 (f_{1,0}(x))  \\ g_{1,1}(\varphi_{1}(x)) = \varphi_1 (f_{1,1}(x))  \end{cases}   x \in [0,1]. 
\end{equation}

The main purpose of this paper is to consider the regularity and singularity of the solutions $\varphi_0$ and $\varphi_1$ in \eqref{eq:gd-dR-fe-def-alt}. 

We can also define a compatible system for a single-indexed family in the same manner; see \cite[Subsection 2.1]{Okamura2024}. 
Assume that $\{f_0, f_1\}$ and $\{g_0,g_1\}$ are both compatible systems and $f_{i,j} = f_j$ and $g_{i,j} = g_j$ for every $i, j \in \{0,1\}$. 
Then, $\varphi_0 = \varphi_1$ on $[0,1]$ and 
\eqref{eq:gd-dR-fe-def-alt} is equivalent to 
\begin{equation}\label{eq:ordinal-dR} 
\begin{cases} g_{0}(\varphi(x)) = \varphi (f_{0}(x))  \\  
g_{1}(\varphi(x)) = \varphi(f_{1}(x))  \end{cases}   
x \in [0,1]. 
\end{equation}
This is the conjugate equation studied in \cite{Okamura2019}, which contains the framework of de Rham functional equations on $[0,1]$. 

We denote the Lebesgue measure by $\ell$. 
We say that a function $\varphi$ on $[0,1]$ is singular if $\varphi^{\prime}(x) = 0$, $\ell$-a.e. \, $x \in (0,1)$.  
We recall the monotone differentiation theorem, which states that every increasing function on $[0,1]$ is Lebesgue measurable and differentiable almost everywhere with respect to $\ell$. 
See \cite[Theorem 1.6.25]{Tao2011}  for more details.

Finally, we state a connection with graph-directed invariant sets. 
Graph-directed invariant sets were first introduced by Mauldin and Williams \cite{MauldinWilliams1988} as a generalization of Hutchinson's self-similar sets. 

Let $V \coloneqq \{0,1\}$. 
For $i, j \in V$, 
let $e_{i,j}$ be a directed edge from $i$ to $j$. 
We remark that $i=j$ is allowed and there are exactly four directed edges. 
See Figure \ref{fig:digraph} below. 

\begin{figure}[H]
\centering
\begin{tikzpicture}[->, >=stealth, auto, node distance=3cm, thick]
  \node[circle, draw] (0) {$0$};
  \node[circle, draw, right of=0] (1) {$1$};

  \draw (0) edge[bend left=25] node[above] {$e_{0,1}$} (1);
  \draw (1) edge[bend left=25] node[below] {$e_{1,0}$} (0);

  \draw (0) edge[loop above] node {$e_{0,0}$} (0);
  \draw (1) edge[loop above] node {$e_{1,1}$} (1);
\end{tikzpicture}
\caption{two-vertex complete digraph with self-loops}\label{fig:digraph}
\end{figure}
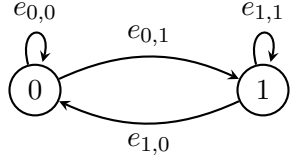

Define a transformation $\Phi_{i,j}$ on $[0,1]^2$, $i, j \in \{0,1\}$,  by 
\[ \Phi_{i,j}(x,y) \coloneqq (f_{i,j}(x), g_{i,j}(y)), \ \ x, y \in [0,1].  \] 
Let 
\begin{equation*}
T(H_0, H_1) \coloneqq \left(  \Phi_{0,0}(H_0) \cup  \Phi_{0,1}(H_1), \Phi_{1,0}(H_0) \cup  \Phi_{1,1}(H_1)  \right), \ H_0, H_1 \subset [0,1]^2. 
\end{equation*}

We denote the graph of $\varphi_i$ by $K_i$, that is, $K_i \coloneqq \left\{(x, \varphi_i (x)) \middle| x \in [0,1] \right\}$, $i = 0,1$. 
Then $T(K_0, K_1) = (K_0, K_1)$.

\section{Affine case}\label{sec:affine}

In this section, 
we consider the case where 
\[ f_{0,0}(x) = p_{0} x, f_{0,1}(x) = (1-p_0) x +p_0, f_{1,0}(x) = p_{1} x, f_{1,1}(x) = (1-p_1) x +p_1, \]
and 
\[ g_{0,0}(x) = q_{0} x, g_{0,1}(x) = (1-q_0) x +q_0, g_{1,0}(x) = q_{1} x, g_{1,1}(x) = (1-q_1) x +q_1, \]
where $p_0, q_0, p_1, q_1 \in (0,1)$. 
These maps are strictly increasing contractions on $[0,1]$. 

\begin{Thm}\label{thm:affine}
(i) If $p_0 \ne q_0$ or $p_1 \ne q_1$, then 
both of the solutions $\varphi_0$ and $\varphi_1$ of  \eqref{eq:gd-dR-fe-def-alt} are singular.\\
(ii) If $p_0 = q_0$ and $p_1 = q_1$, then $\varphi_i (x) = x, x \in [0,1]$, $i = 0,1$. 
\end{Thm}

Set 
\[ p_{i,0} \coloneqq p_i, \, p_{i,1} \coloneqq 1 - p_i, \  \ q_{i,0} \coloneqq q_i, \, q_{i,1} \coloneqq 1 - q_i, \ i = 0,1. \]

\begin{proof}
(ii) is obvious, so we show (i). 
Assume that $x\in [0,1] \setminus D^{f}_0$ and $\varphi_0^{\prime}(x) > 0$.
Then, by Section \ref{sec:frame},  there exists a unique sequence $(i_n)_n$ in $\{0,1\}^{\mathbb N}$ such that $x \in I^{f}_{0} (i_1, \dots, i_n)$ for each $n \ge 1$. 
Then
\[ \lim_{n \to \infty} \frac{|I^{g}_{0} (i_1, \dots, i_n)|}{|I^{f}_{0} (i_1, \dots, i_n)|}  \]
\[ = \lim_{n \to \infty} \frac{g_{0, i_1} \circ g_{i_1, i_2} \circ \cdots \circ g_{i_{n-1}, i_n} (1) - g_{0, i_1} \circ g_{i_1, i_2} \circ \cdots \circ g_{i_{n-1}, i_n} (0)}{f_{0, i_1} \circ f_{i_1, i_2} \circ \cdots \circ f_{i_{n-1}, i_n} (1) - f_{0, i_1} \circ f_{i_1, i_2} \circ \cdots \circ f_{i_{n-1}, i_n} (0)} \]
\[ = \lim_{n \to \infty} \frac{\varphi_0 \left( f_{0, i_1} \circ f_{i_1, i_2} \circ \cdots \circ f_{i_{n-1}, i_n} (1) \right) - \varphi_0 \left( f_{0, i_1} \circ f_{i_1, i_2} \circ \cdots \circ f_{i_{n-1}, i_n} (0) \right)}{f_{0, i_1} \circ f_{i_1, i_2} \circ \cdots \circ f_{i_{n-1}, i_n} (1) - f_{0, i_1} \circ f_{i_1, i_2} \circ \cdots \circ f_{i_{n-1}, i_n} (0)} \]
\[ = \varphi_0^{\prime}(x) > 0.  \]
Since each map is affine, 
\[ \frac{|I^{g}_{0} (i_1, \dots, i_n)|}{|I^{f}_{0} (i_1, \dots, i_n)|} = \frac{q_{0 i_1}}{p_{0,i_1}} \frac{q_{i_1, i_2}}{p_{i_1, i_2}} \cdots \frac{q_{i_{n-1}, i_{n}}}{p_{i_{n-1}, i_n}}. \]
Hence 
\begin{equation}\label{eq:lim-ratio}
\lim_{n \to \infty} \frac{q_{i_{n} i_{n+1}}}{p_{i_{n} i_{n+1}}} = 1. 
\end{equation}

Set 
\[ \mathscr{R} \coloneqq \left\{\frac{q_{0}}{p_{0}}, \frac{1-q_{0}}{1-p_{0}}, \frac{q_{1}}{p_{1}}, \frac{1 - q_{1}}{1 - p_{1}}   \right\}. \]
Then for every $n \ge 1$, 
\[ \frac{q_{i_{n} i_{n+1}}}{p_{i_{n} i_{n+1}}}  \in \mathscr{R}. \]

We consider the case where $p_0 \ne q_0$ and $p_1 \ne q_1$. 
Then no element of $\mathscr{R}$ equals $1$. 
Hence \eqref{eq:lim-ratio} fails.

Therefore, it holds that  for every $x\in [0,1] \setminus D^{f}_0$ where $\varphi_0$ is differentiable, 
the derivative of $\varphi_0$ at $x$ is $0$. 

We consider the case where $p_0 = q_0$ and $p_1 \ne q_1$. 
Let $\mathcal{A}_0$ be the set of $(i_n)_n$ such that there exists $N$ such that for every $n \ge N$, $i_n = 0$.
If \eqref{eq:lim-ratio} holds, then there exists $N$ such that for every $n \ge N$, $(i_n, i_{n+1}) = (0,0)$ or $(i_n, i_{n+1}) = (0,1)$, and hence, $(i_n)_n \in \mathcal{A}_0$. 
Then 
\[ x \in \bigcup_{(i_n)_n \in \mathcal{A}_0} \bigcap_{n \ge 1} I^{f}_{0} (i_1, \dots, i_n).  \]
Since $0 < p_{0,0} < 1$, it holds that 
\[ \ell \left( I^{f}_{0} (i_1, \dots, i_n) \right) = \left| I^{f}_{0} (i_1, \dots, i_n) \right| = p_{0, i_1} p_{i_1, i_2} \cdots p_{i_{n-1}, i_n} \to 0, \ n \to \infty. \]
Since $\mathcal{A}_0$ is countable,  
we see that $\displaystyle \ell \left( \bigcup_{(i_n)_n \in \mathcal{A}_0} \bigcap_{n \ge 1} I^{f}_{0} (i_1, \dots, i_n)  \right) = 0$.

Therefore, it holds that  for $\ell$-a.e. $x\in [0,1] \setminus D^{f}_0$ where $\varphi_0$ is differentiable, 
the derivative of $\varphi_0$ at $x$ is $0$. 

The case where $p_0 \ne q_0$ and $p_1 = q_1$ can also be considered in the same manner. 
Thus we see that $\varphi_0$ is singular. 

In the same manner, we can show that $\varphi_1$ is singular. 
\end{proof}

\section{Linear fractional case}\label{sec:LF}

We say that a square matrix $A = \begin{pmatrix} a & b \\  c & d \end{pmatrix}$ is in the class $\mathscr{M}$ if 
$0 < a+b \le c+d$, $d > b \ge 0$, 
$ad - bc > 0$ and $\sqrt{ad-bc} \le \min\{d, c+d\}$. 
Let $\Phi(A;x) \coloneqq \dfrac{ax+b}{cx+d}$, which is the linear fractional transformation associated with $A$. 
If $A \in \mathscr{M}$, then, by \cite[Lemma 2.9]{Okamura2024}, 
$\Phi(A;x)$ is a strictly increasing weak contraction with some comparison function on $[0,1]$. 

We say that a compatible system $\{h_{0}, h_{1}\}$ is an {\it LF system} if there exist $A_0, A_1 \in \mathscr{M}$ such that $h_i (x) = \Phi(A_i; x)$, $i = 0,1$. 
See \cite[Subsection 2.3]{Okamura2024} for more details. 

In this section, 
we consider the case where 
\begin{equation}\label{eq:def-f-dyadic} 
f_{0,0}(x) = f_{1,0}(x) = \frac{x}{2}, \ \ f_{0,1}(x) = f_{1,1}(x) = \frac{x+1}{2}, 
\end{equation}
and each of $\{g_{0,0}, g_{0,1}\}$ and $\{g_{1,0}, g_{1,1}\}$ is an LF system. 
Assume that $g_{i,j}(x) = \Phi(A_{i,j}; x)$ for $A_{i,j} \in \mathscr{M}$. 
Let $A_{i,j} =  \begin{pmatrix} a_{i,j} & b_{i,j} \\  c_{i,j} & d_{i,j} \end{pmatrix}$. 
By the definition of LF systems, 
it holds that 
\begin{equation}\label{eq:cd-posi} 
\min_{x \in [0,1]} c_{i,j} x + d_{i,j} = \min\{d_{i,j}, c_{i,j} + d_{i,j}\} \ge \sqrt{a_{i,j} d_{i,j} - b_{i,j} c_{i,j}} > 0, \ i, j \in \{0,1\}, 
\end{equation} 
and 
\begin{equation}\label{eq:compatible-LF} 
0 = b_{i,0} < \dfrac{a_{i,0}}{c_{i,0} + d_{i,0}} = \dfrac{b_{i,1}}{d_{i,1}} < \frac{a_{i,1} + b_{i,1}}{c_{i,1} + d_{i,1}} = 1, \ \ i = 0,1.
\end{equation}

Let 
\begin{equation}\label{eq:def-alpha} 
\alpha_i \coloneqq \frac{c_{i,0} + d_{i,0}}{a_{i,0}} - 2 = \frac{d_{i,1}}{b_{i,1}} - 2, \ i = 0,1. 
\end{equation} 

If $c_{i,j} = 0$, then $g_{i,j}$ is an affine map. 
Therefore, the framework here contains the framework of Section \ref{sec:affine} with $p_0 = p_1 = 1/2$. 

We denote the transpose of a square matrix $A$ by $A^{\top}$. 
We remark that $\Phi \left(A_{i,j}^{\top}; \alpha_i \right)$ is well-defined. 
It suffices to show that $b_{i,j} \alpha_i + d_{i,j} \ne 0$. 
This is obvious if $j=0$ since $b_{i,j} = 0$ and $d_{i,j} > 0$.
If $j=1$, then $b_{i,1} \alpha_i + d_{i,1} = 2(d_{i,1} - b_{i,1}) > 0$. 

\begin{Thm}\label{thm:LF}
(i) If $\Phi \left(A_{i,j}^{\top}; \alpha_i \right) \ne \alpha_j$ holds for some $i, j \in \{0,1\}$, 
then, both of the solutions $\varphi_0$ and $\varphi_1$ of  \eqref{eq:gd-dR-fe-def-alt} are singular.\\
(ii) If $\Phi \left(A_{i,j}^{\top}; \alpha_i \right) = \alpha_j$ holds for every $i, j \in \{0,1\}$, 
then 
\begin{equation}\label{eq:solution-ac} 
\varphi_0 (x) = \frac{x}{-c_{0,0} x + 1 + c_{0,0}} \textup{ and } \varphi_1 (x) = \frac{x}{-c_{1,1} x + 1 + c_{1,1}}. 
\end{equation}
In particular, both $\varphi_0$ and $\varphi_1$ are smooth. 
\end{Thm}

\begin{proof}
(i) We show that $\varphi_0$ is singular. 

We first remark that  for every $y\in [0,1] \setminus D^{f}_0$ there exists a unique sequence $(i_n)_n$ in $\{0,1\}^{\mathbb N}$ such that $y \in I^{f}_{0} (i_1, \dots, i_n)$ for each $n \ge 1$. 

Assume that $x\in [0,1] \setminus D^{f}_0$ and $\varphi_0^{\prime}(x) > 0$.
Let $(i_n)_n$ be the unique sequence in $\{0,1\}^{\mathbb N}$ such that $x \in I^{f}_{0} (i_1, \dots, i_n)$ for each $n \ge 1$. 
Let 
\[ \begin{pmatrix} p_n & q_n \\  r_n & s_n \end{pmatrix} = A_{0,i_1} A_{i_1, i_2} \cdots A_{i_{n-1}, i_n}. \]
By the definition of $g_{i,j}$, it holds that 
\[ g_{0, i_1} \circ g_{i_1, i_2} \circ \cdots \circ g_{i_{n-1}, i_n} (x) = \Phi \left(A_{0,i_1} A_{i_1, i_2} \cdots A_{i_{n-1}, i_n}; x\right) = \frac{p_n x + q_n}{r_n x + s_n}, \ x \in [0,1]. \]

We show that $\min\{s_n, r_n + s_n\} > 0$ by induction on $n$. 
The case where $n=1$ is obvious. 
Assume that $\min\{s_n, r_n + s_n\} > 0$. 
Then $\displaystyle \min_{x \in [0,1]} r_n x + s_n = \min\{s_n, r_n + s_n\} > 0$. 
By \eqref{eq:cd-posi}, it holds that $d_{i_n, i_{n+1}} > 0$ and $c_{i_n, i_{n+1}} + d_{i_n, i_{n+1}} > 0$. 
By \eqref{eq:compatible-LF}, we see that 
$ \dfrac{b_{i_n, i_{n+1}}}{d_{i_n, i_{n+1}}} \in [0,1]$ and $\dfrac{a_{i_n, i_{n+1}} + b_{i_n, i_{n+1}}}{c_{i_n, i_{n+1}} + d_{i_n, i_{n+1}}}  \in [0,1]$. 
Since 
\[ (s_{n+1}, r_{n+1} + s_{n+1}) \]
\begin{equation}\label{eq:rs-iteration} 
 = \left(b_{i_n, i_{n+1}} r_{n} + d_{i_n, i_{n+1}} s_n,  (a_{i_n, i_{n+1}} + b_{i_n, i_{n+1}}) r_{n} + (c_{i_n, i_{n+1}} + d_{i_n, i_{n+1}}) s_n \right), 
\end{equation}
we obtain that 
\[ s_{n+1} = d_{i_n, i_{n+1}} \left(r_n \frac{b_{i_n, i_{n+1}}}{d_{i_n, i_{n+1}}} + s_n \right) > 0 \]
and 
\[ r_{n+1} + s_{n+1} = (c_{i_n, i_{n+1}} + d_{i_n, i_{n+1}}) \left(r_n \frac{a_{i_n, i_{n+1}} + b_{i_n, i_{n+1}}}{c_{i_n, i_{n+1}} + d_{i_n, i_{n+1}}} + s_n \right) > 0. \]

Since $b_{i,0} = 0 < d_{i,0}$, $ \Phi \left(A_{i, 0}^{\top} ; x\right)$ is well-defined for every $x \in \mathbb{R}$ and $i=0,1$. 
Since $\frac{r_n}{s_n} > -1 > -\frac{d_{i,1}}{b_{i,1}}$, $ \Phi \left(A_{i, 1}^{\top}; \frac{r_n}{s_n}\right)$ is well-defined for $i=0,1$ and $n \ge 1$. 
Therefore, we obtain that for every $n \ge 1$, 
\begin{equation}\label{eq:r-over-s} 
\frac{r_{n+1}}{s_{n+1}} = \Phi \left(A_{i_n, i_{n+1}}^{\top} ; \frac{r_n}{s_n} \right)
\end{equation}

By the definition of $f_{i,j}$ in \eqref{eq:def-f-dyadic}, it holds that $D^{f}_{0}$ is the set of dyadic rationals on $[0,1]$ and 
\[ \frac{|I^{g}_{0} (i_1, \dots, i_n)|}{|I^{f}_{0} (i_1, \dots, i_n)|} = 2^n \left( \frac{p_n + q_n}{r_n + s_n} - \frac{q_n}{s_n} \right) = 2^n \frac{p_n s_n - q_n r_n}{s_n (r_n + s_n)}. \]
Since 
$\displaystyle \lim_{n \to \infty} \frac{|I^{g}_{0} (i_1, \dots, i_n)|}{|I^{f}_{0} (i_1, \dots, i_n)|}  = \varphi_0^{\prime}(x) > 0$, 
we obtain that 
\[ \lim_{n \to \infty} \frac{|I^{g}_{0} (i_1, \dots, i_{n+1})| / |I^{f}_{0} (i_1, \dots, i_{n+1})| }{|I^{g}_{0} (i_1, \dots, i_n)| / |I^{f}_{0} (i_1, \dots, i_n)| } = 1. \]
This is equivalent to  
\[ \lim_{n \to \infty} \frac{\det(A_{i_n, i_{n+1}}) s_n (r_n + s_n)}{s_{n+1} (r_{n+1} + s_{n+1})} = \frac{1}{2}. \]

By \eqref{eq:rs-iteration}, this is equivalent to  
\[ \lim_{n \to \infty} t_n = \frac{1}{2}, \]
where we let 
\[ t_n \coloneqq \frac{\det(A_{i_n, i_{n+1}}) (\frac{r_n}{s_n} + 1)}{(b_{i_n, i_{n+1}} \frac{r_{n}}{s_n} + d_{i_n, i_{n+1}}) ((a_{i_n, i_{n+1}} + b_{i_n, i_{n+1}}) \frac{r_{n}}{s_n} + c_{i_n, i_{n+1}} + d_{i_n, i_{n+1}})}. \]
If $i_{n+1} = 0$, then $b_{i_n, i_{n+1}} = 0$ and hence 
\begin{equation}\label{eq:rs-zero} 
t_n = \frac{\frac{r_n}{s_n} + 1}{\frac{r_n}{s_n} + \frac{c_{i_n, 0} + d_{i_n, 0}}{a_{i_n, 0}}}. 
\end{equation} 
If $i_{n+1} = 1$, then $a_{i_n, i_{n+1}} + b_{i_n, i_{n+1}} = c_{i_n, i_{n+1}} + d_{i_n, i_{n+1}} > 0$ and hence 
\begin{equation}\label{eq:rs-one} 
t_n = \frac{\frac{d_{i_n,1}}{b_{i_n,1}} - 1}{\frac{r_n}{s_n} + \frac{d_{i_n, 1}}{b_{i_n, 1}}}. 
\end{equation}

For $y \in [0,1] \setminus  D_0^f$, let 
\[ \mathbb{N}_{i,j}(y) \coloneqq \left\{n \in \mathbb{N} \colon (i_n, i_{n+1}) = (i,j) \right\}, \ i, j \in \{0,1\},  \]
and 
\[ \mathbb{N}_{i}(y) \coloneqq \mathbb{N}_{i,0}(y) \cup \mathbb{N}_{i,1}(y), \ i \in \{0,1\}.  \]
Then $\{\mathbb{N}_{i,j}(y)\}_{i,j}$ are disjoint and $\mathbb{N} = \mathbb{N}_{0}(y) \cup \mathbb{N}_{1}(y)$. 
Let $\mathcal{N}$ be the set of $y \in [0,1] \setminus  D_0^f$ such that $\mathbb{N}_{i,j}(y)$ is finite for some $(i, j)$. 

Assume that $x \in [0,1] \setminus  \mathcal{N}$. 
Recall \eqref{eq:def-alpha}. 
Then, 
by \eqref{eq:rs-zero} and \eqref{eq:rs-one}, 
\[ \lim_{n \to \infty; n \in \mathbb{N}_i (x)} \frac{r_n}{s_n} = \alpha_i, \ i \in \{0,1\}.  \]
By this and \eqref{eq:r-over-s}, 
it holds that 
$\Phi \left(A_{i,j}^{\top}; \alpha_i \right) = \alpha_j$ for every $i, j \in \{0,1\}$. 
By the assumption, $x \in \mathcal{N}$. 
By the Borel normal number theorem, $\ell(\mathcal{N}) = 0$. 

Thus we see that $\varphi_0$ is singular. 
In the same manner, we also see that $\varphi_1$ is singular. 

(ii) We first remark that $\Phi(A;x) = \Phi(cA;x)$ for every $c \ne 0$ and every square matrix $A$. 
Since $a_{0,0}, a_{1,0}, b_{0,1}$ and $b_{1,1}$ are all positive, 
we can assume that $a_{0,0} = a_{1,0} = b_{0,1} = b_{1,1} = 1$. 

Since $\alpha_0 = c_{0,0} + d_{0,0} - 2$, 
the equation $\Phi \left(A_{0,0}^{\top}; \alpha_0 \right) = \alpha_0$ is equivalent to  the equation $(c_{0,0} + d_{0,0} -1)(d_{0,0} - 2) = 0$. 
Since $c_{0,0} + d_{0,0} > a_{0,0} = 1$,  we obtain that $d_{0,0} = 2$. 
By this and $b_{0,0} = 0$, 
we obtain that 
$A_{0,0} = \begin{pmatrix} 1 & 0 \\ c_{0,0} & 2 \end{pmatrix}$.  
Hence, 
$g_{0,0}(x) = \dfrac{x}{c_{0,0} x + 2}$.

Since $\alpha_1 =  d_{1,1} - 2$, 
the equation $\Phi \left(A_{1,1}^{\top}; \alpha_1 \right) = \alpha_1$ is equivalent to  the equation $(c_{1,1} - d_{1,1} + 2)(d_{1,1} - 1) = 0$. 
Since $d_{1,1} > b_{1,1} = 1$, we obtain that $d_{1,1} = c_{1,1} + 2$. 
By this and $a_{1,1} + 1 = c_{1,1} + d_{1,1}$, 
we obtain that 
$A_{1,1} = \begin{pmatrix} 2c_{1,1} + 1 & 1 \\ c_{1,1} & c_{1,1} + 2 \end{pmatrix}$.  
Hence, 
$g_{1,1}(x) = \dfrac{(2c_{1,1} + 1)x + 1}{c_{1,1} x + c_{1,1} + 2}$. 

It also holds that $\alpha_i = c_{i,i}, \, i = 0,1$.

We see that the equation $\Phi \left(A_{0,1}^{\top}; \alpha_0 \right) = \alpha_1$ is equivalent to  the equation $(c_{0,0} + 1)c_{0,1} + c_{0,0} (d_{0,1} - 1) = 2c_{1,1} (c_{0,0} + 1)$. 
By \eqref{eq:compatible-LF}, 
it holds that $d_{0,1} = c_{0,0} + 2$.  
Hence $c_{0,1} = 2 c_{1,1} - c_{0,0}$. 
Hence $a_{0,1} = c_{0,1} + d_{0,1} - 1 = 2c_{1,1} + 1$. 
Thus we obtain that $A_{0,1} = \begin{pmatrix} 2c_{1,1} + 1 & 1 \\ 2c_{1,1} - c_{0,0} & c_{0,0} + 2 \end{pmatrix}$.  
Hence, 
$g_{0,1}(x) = \dfrac{(2c_{1,1} + 1)x + 1}{(2c_{1,1}-c_{0,0}) x + c_{0,0} + 2}$. 

We see that the equation $\Phi \left(A_{1,0}^{\top}; \alpha_1 \right) = \alpha_0$ is equivalent to  the equation $c_{1,0} - c_{0,0} d_{1,0} = -c_{1,1}$. 
By \eqref{eq:compatible-LF}, 
it holds that $c_{1,0} + d_{1,0} = c_{1,1} + 2$.  
By solving the system of linear equations, we obtain that 
\[ c_{1,0} = \frac{c_{0,0} c_{1,1} + 2 c_{0,0} - c_{1,1}}{c_{0,0} + 1} \textup{ and }  d_{1,0} = \frac{2(c_{1,1} + 1)}{c_{0,0} + 1}. \]
By this and $b_{1,0} = 0$, 
we obtain that 
$A_{1,0} = \begin{pmatrix} 1 & 0 \\  \frac{c_{0,0} c_{1,1} + 2 c_{0,0} - c_{1,1}}{c_{0,0} + 1}  & \frac{2(c_{1,1} + 1)}{c_{0,0} + 1} \end{pmatrix}$.  
Hence, 
$g_{1,0}(x) = \dfrac{(c_{0,0} + 1)x}{(c_{0,0} c_{1,1} + 2 c_{0,0} - c_{1,1}) x + 2(c_{1,1} + 1)}$. 

Now by computations, we see that the functions $\varphi_0$ and $\varphi_1$ defined in \eqref{eq:solution-ac} satisfy \eqref{eq:gd-dR-fe-def-alt}.  
\end{proof}

\begin{Rem}
(i) The case where $A_{0,0} = A_{1,0}$ and $A_{0,1} = A_{1,1}$ implies the regularity and singularity results of \cite[Theorem 1.2]{Okamura2014}. \\
(ii) Assume that $\Phi \left(A_{i,j}^{\top}; \alpha_i \right) = \alpha_j$ holds for every $i, j \in \{0,1\}$. 
Then the matrices $A_{i,j}, i, j \in \{0,1\}$, are determined by $(c_{0,0}, c_{1,1})$. 
The inequalities in \eqref{eq:cd-posi} hold for every $i, j \in \{0,1\}$ if and only if the following conditions hold: 
\begin{equation}\label{eq:complex-1} 
\begin{cases} c_{0,0} \ge \sqrt{2} - 2 \\ c_{1,1} \le \sqrt{2} \\ c_{0,0} \le 2c_{1,1} + 1 \\ 2(c_{0,0} + 1)(c_{1,1} + 1) \le (c_{0,0} + 2)^2 \\ 2(c_{1,1} + 1) \le (c_{0,0} + 1) (c_{1,1} + 2)^2  \end{cases}. 
\end{equation}

The resulting region is rather complicated and is shown in the following figure: 
\begin{figure}[H]
\centering
\includegraphics[scale=0.3]{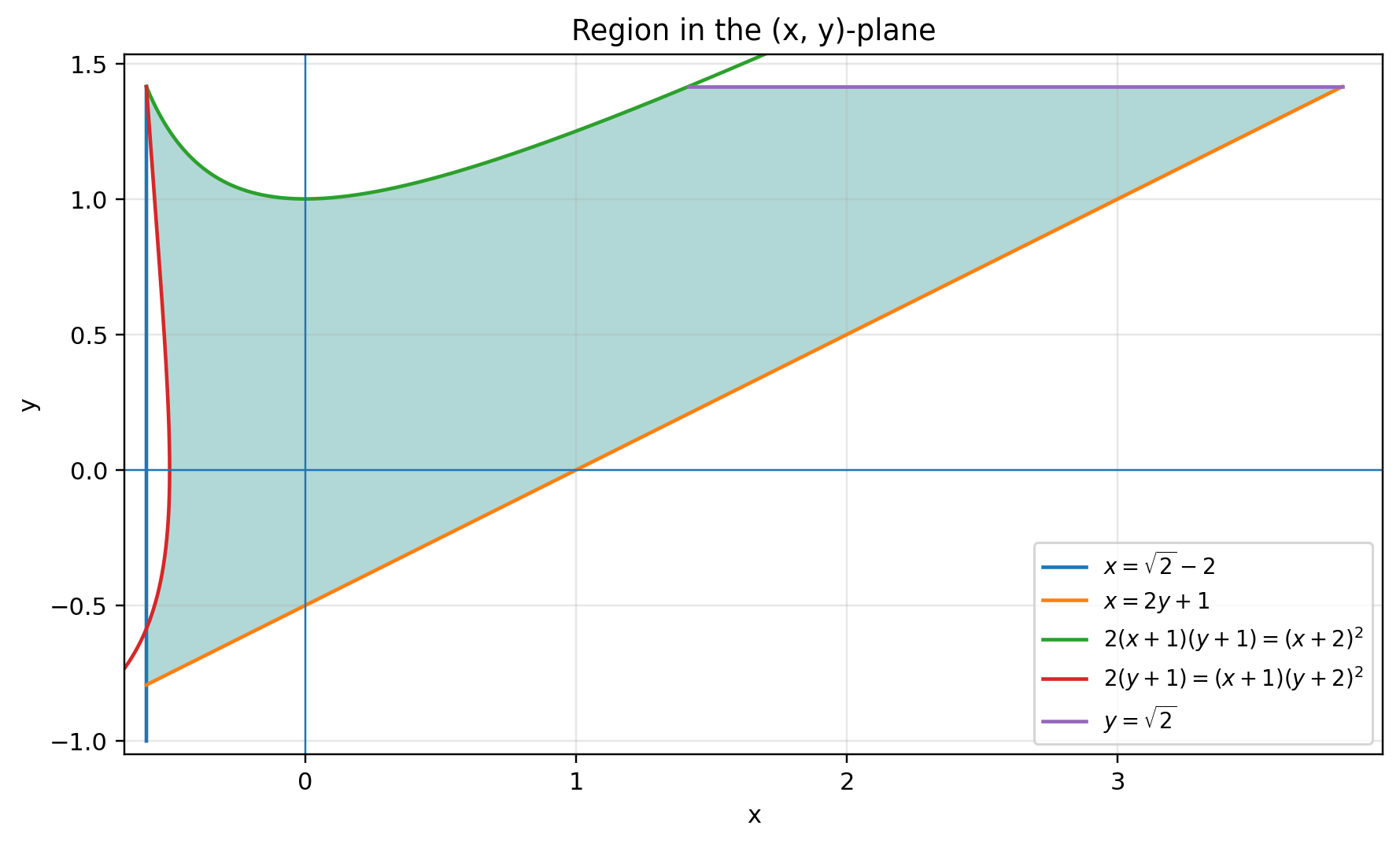}
\caption{\footnotesize admissible region in the $(c_{0,0}, c_{1,1})$-plane}
\end{figure}

By symmetry, the function $1- g_{1,1}(1-x)$ corresponds to $g_{0,0}(x)$ and it holds that 
\[ 1- g_{1,1}(1-x) = \frac{x}{-\frac{c_{1,1}}{c_{1,1} + 1} x + 2}. \]
Thus $c_{1,1}^{\prime} \coloneqq - \frac{c_{1,1}}{c_{1,1} + 1}$ corresponds to $c_{0,0}$. 
This map is an involution, specifically, $c_{1,1} = - \frac{c_{1,1}^{\prime}}{c_{1,1}^{\prime} + 1}$. 
Then \eqref{eq:complex-1} is equivalent to  
\begin{equation}\label{eq:complex-2} 
\begin{cases} c_{0,0} \ge \sqrt{2} - 2 \\ c_{1,1}^{\prime} \ge \sqrt{2} -2 \\ (c_{0,0} + 1)(c_{1,1}^{\prime} + 1) \le 2 \\ c_{0,0} \ge 2\frac{c_{1,1}^{\prime} + 1}{(c_{1,1}^{\prime} + 2)^2} -1  \\ c_{1,1}^{\prime} \ge 2\frac{c_{0,0} + 1}{(c_{0,0} + 2)^2} -1  \end{cases}. 
\end{equation}

The corresponding region for \eqref{eq:complex-2} is shown in the following figure: 
\begin{figure}[H]
\centering
\includegraphics[scale=0.3]{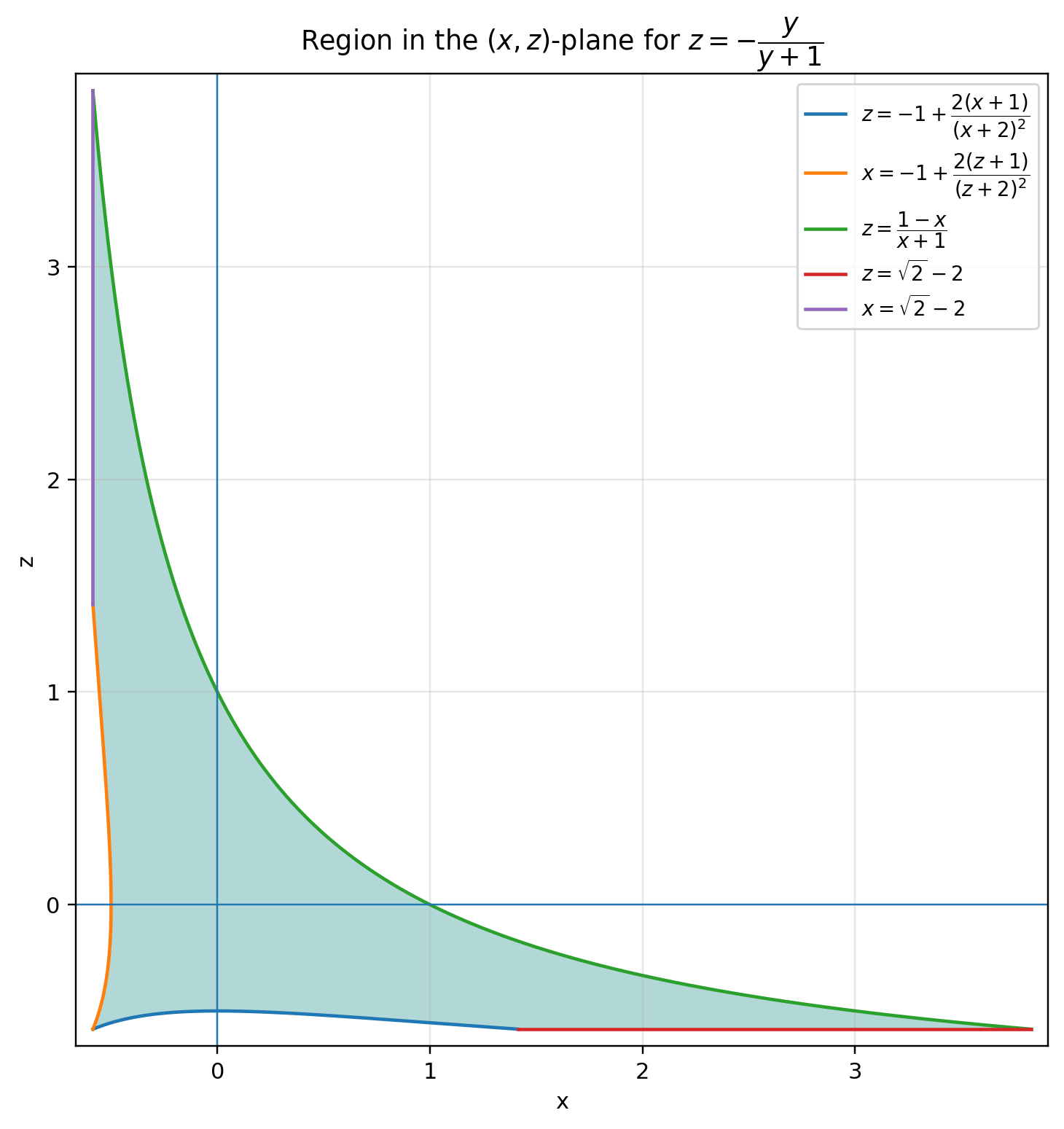}
\caption{\footnotesize admissible region in the $(c_{0,0}, c_{1,1}^{\prime})$-plane}
\end{figure}
This domain  is symmetric with respect to $x=z$. 

If $c_{0,0} = c_{1,1}$, then $A_{0,0} = A_{1,0}$ and $A_{0,1} = A_{1,1}$. 
Furthermore, $\varphi_0 (x) = \varphi_1 (x)$ for every $x \in [0,1]$. 
\end{Rem}

\section{Non-linear case}\label{sec:nl}

In this section, we consider the case where \eqref{eq:def-f-dyadic} holds and each $\{g_{i,0}, g_{i,1}\}$ is merely assumed to be a compatible system, $i = 0,1$. 
For $h \colon [0,1] \to \mathbb{R}$, 
let $\|h\|_{\textup{Lip}} \coloneqq \sup\left\{ \frac{|h(x) - h(y)|}{|x-y|} \colon x \ne y  \right\}$. 
The following result is similar to \cite[Theorem 7.3]{Hata1985}. 

\begin{Thm}\label{thm:nl}
If $\prod_{i,j \in \{0,1\}} \left\| g_{i,j} \right\|_{\textup{Lip}} < \frac{1}{16}$, then both of the solutions $\varphi_0$ and $\varphi_1$ of  \eqref{eq:gd-dR-fe-def-alt} are singular. 
\end{Thm}

\begin{proof}
We show that $\varphi_0$ is singular. 
Define $\mathbb{N}_{i,j}(y)$ as in the proof of Theorem \ref{thm:LF}. 
Let $\mathcal{S}$ be the set of $y \in [0,1] \setminus  D^f_0$ such that 
\begin{equation}\label{eq:pattern-Borel-normal-lim}
\lim_{N \to \infty} \frac{\left|\mathbb{N}_{i,j}(y) \cap \{1, \dots, N\}\right|}{N} = \frac{1}{4}
\end{equation}
for every $i, j \in \{0,1\}$.

Assume that $x\in [0,1] \setminus D^{f}_0$ and $\varphi_0^{\prime}(x) > 0$.
Let $(i_n)_n$ be the unique sequence in $\{0,1\}^{\mathbb N}$ such that $x \in I^{f}_{0} (i_1, \dots, i_n)$ for each $n \ge 1$. 
We remark that $\| g_{i,j} \|_{\textup{Lip}} \le 1$. 
We see that for every $N \ge 1$, 
\[ \left|I^{g}_{0} (i_1, \dots, i_N)\right| \le \prod_{k=1}^{N-1} \left\| g_{i_k, i_{k+1}} \right\|_{\textup{Lip}} = \prod_{i,j \in \{0,1\}} \left\| g_{i,j} \right\|_{\textup{Lip}}^{N_{i,j}} \]
where we let $N_{i,j} \coloneqq \left|\mathbb{N}_{i,j}(x) \cap \{1, \dots, N-1\}\right|$. 
Assume that $x \in \mathcal{S}$. 
Then by the assumption and \eqref{eq:pattern-Borel-normal-lim}, 
\[ \lim_{N \to \infty} 2^N  \left|I^{g}_{0} (i_1, \dots, i_N)\right| = 0. \]
On the other hand, 
\[ \lim_{N \to \infty} 2^N  \left|I^{g}_{0} (i_1, \dots, i_N)\right| = \lim_{N \to \infty} \frac{|I^{g}_{0} (i_1, \dots, i_N)|}{|I^{f}_{0} (i_1, \dots, i_N)|}   = \varphi_0^{\prime}(x)  > 0. \]
Hence $x \in [0,1] \setminus  \mathcal{S}$. 
By the Borel normal number theorem, $\ell(\mathcal{S}) = 1$. 
Thus we see that $\varphi_0$ is singular. 
In the same manner, we can show that $\varphi_1$ is singular. 
\end{proof}

The above proof is partly different from the proof of \cite[Theorem 7.3]{Hata1985}, since we can apply the monotone differentiation theorem. 

\section{Examples}\label{sec:exa}

In this section, we give four examples with graphs of the solutions. 

\begin{Exa}\label{exa:affine}
Let 
\[ f_{0,0}(x) =  \frac{x}{2}, \ f_{0,1}(x) = \frac{x +1}{2}, \ \ f_{1,0}(x) = \frac{x}{3}, \ f_{1,1}(x) = \frac{2x + 1}{3}, \]
and 
\[ g_{0,0}(x) = \frac{x}{4}, \ g_{0,1}(x) = \frac{3x + 1}{4}, \ \  g_{1,0}(x) = \frac{x}{5}, \ g_{1,1}(x) = \frac{4x +1}{5}. \]
By Theorem \ref{thm:affine}, the solutions $\varphi_0$ and $\varphi_1$ are both singular. 
The graphs of the solutions are given as follows: 
\begin{figure}[H]
\begin{minipage}{0.49\columnwidth}
\centering
\includegraphics[scale=0.4]{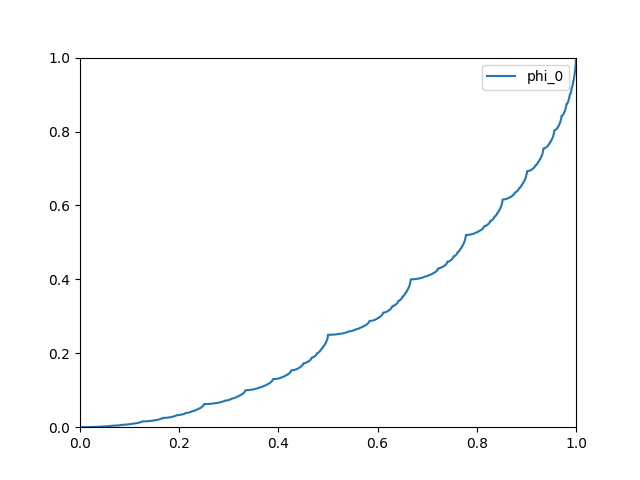}
\caption{\footnotesize Graph of $\varphi_0$}
\end{minipage}
\begin{minipage}{0.49\columnwidth}
\centering
\includegraphics[scale=0.4]{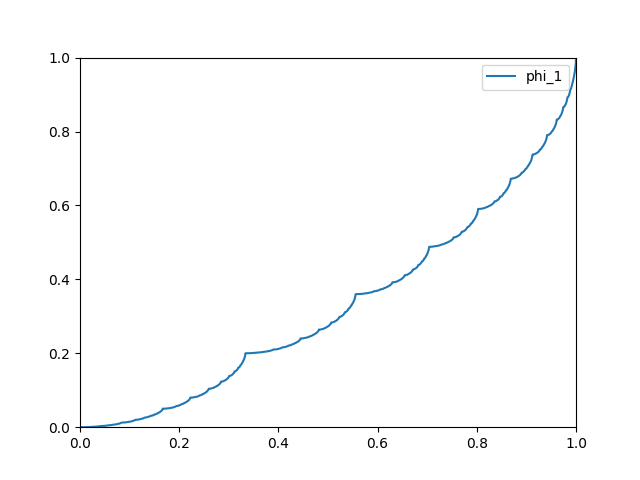}
\caption{\footnotesize Graph of $\varphi_1$}
\end{minipage}
\end{figure}
\end{Exa}

\begin{Exa}\label{exa:LF-sing}
Let $\{f_{i,j}\}_{i,j}$ be the functions as in \eqref{eq:def-f-dyadic}. 
Let 
\[ g_{0,0}(x) = \frac{x}{x+1}, \ g_{0,1}(x) = \frac{1}{2-x}, \ \ g_{1,0}(x) = \frac{x}{3-x}, \ g_{1,1}(x) = \frac{3x +1}{2x+2}. \]
By Theorem \ref{thm:LF} (i), the solutions $\varphi_0$ and $\varphi_1$ are both singular. 
The graphs of the solutions are given as follows: 
\begin{figure}[H]
\begin{minipage}{0.49\columnwidth}
\centering
\includegraphics[scale=0.4]{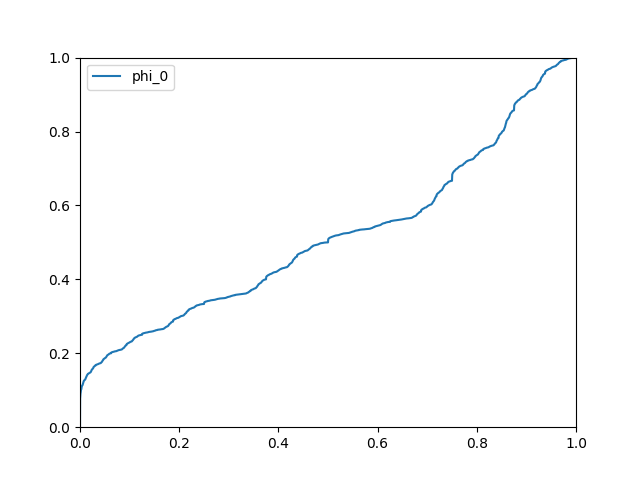}
\caption{\footnotesize Graph of $\varphi_0$}
\end{minipage}
\begin{minipage}{0.49\columnwidth}
\centering
\includegraphics[scale=0.4]{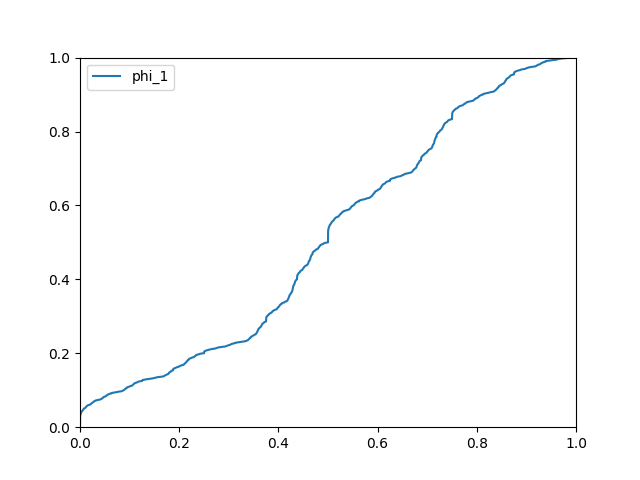}
\caption{\footnotesize Graph of $\varphi_1$}
\end{minipage}
\end{figure}
\end{Exa}

\begin{Exa}\label{exa:LF-ac}
Let $\{f_{i,j}\}_{i,j}$ be the functions as in \eqref{eq:def-f-dyadic}. 
Assume that $\Phi \left(A_{i,j}^{\top}; \alpha_i \right) = \alpha_j$ holds for every $i, j \in \{0,1\}$. 
Let $c_{0,0} = -1/2$ and $c_{1,1} = 1/2$. 
Then 
\[ g_{0,0}(x) = \frac{2x}{4-x}, \ g_{0,1}(x) = \frac{4x+2}{3x+3}, \ \ g_{1,0}(x) = \frac{2x}{12- 7x}, \ g_{1,1}(x) = \frac{4x +2}{x+5}. \]
By Theorem \ref{thm:LF} (ii), the solutions $\varphi_0$ and $\varphi_1$ are given by 
\[ \varphi_0 (x) = \frac{2x}{x+1}, \ \varphi_1 (x) = \frac{2x}{3-x}. \] 
The graphs of the solutions are given as follows: 
\begin{figure}[H]
\begin{minipage}{0.49\columnwidth}
\centering
\includegraphics[scale=0.4]{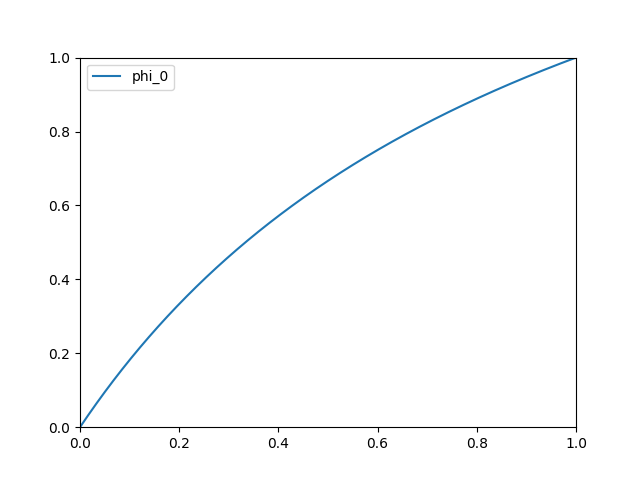}
\caption{\footnotesize Graph of $\varphi_0$}
\end{minipage}
\begin{minipage}{0.49\columnwidth}
\centering
\includegraphics[scale=0.4]{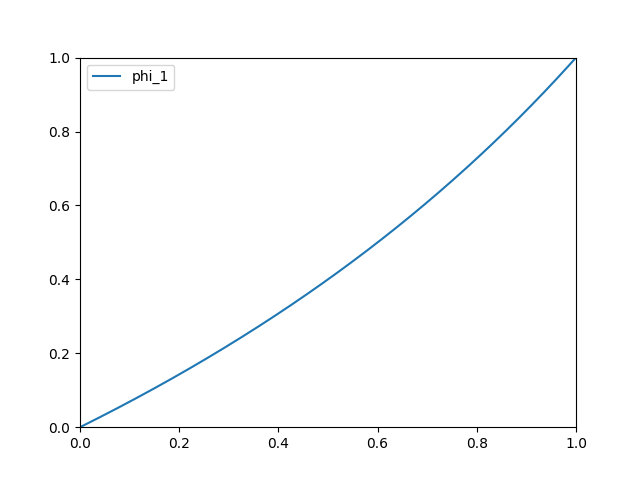}
\caption{{\footnotesize Graph of $\varphi_1$}}
\end{minipage}
\end{figure}
\end{Exa}

\begin{Exa}\label{exa:nl}
Let $\{f_{i,j}\}_{i,j}$ be the functions as in \eqref{eq:def-f-dyadic}. 
Let  
\[ g_{0,0}(x) = \frac{x^2}{x+1}, \ g_{0,1}(x) = \frac{x+1}{2}, \ \ g_{1,0}(x) = \frac{x^{3/2}}{8}, \ g_{1,1}(x) = \frac{7x +1}{8}. \]
In this case, it holds that $\prod_{i,j \in \{0,1\}} \| g_{i,j} \|_{\textup{Lip}}  = \frac{63}{1024} < \frac{1}{16}$. 
By Theorem \ref{thm:nl}, the solutions $\varphi_0$ and $\varphi_1$ are singular. 
The graphs of the solutions are given as follows: 
\begin{figure}[H]
\begin{minipage}{0.49\columnwidth}
\centering
\includegraphics[scale=0.4]{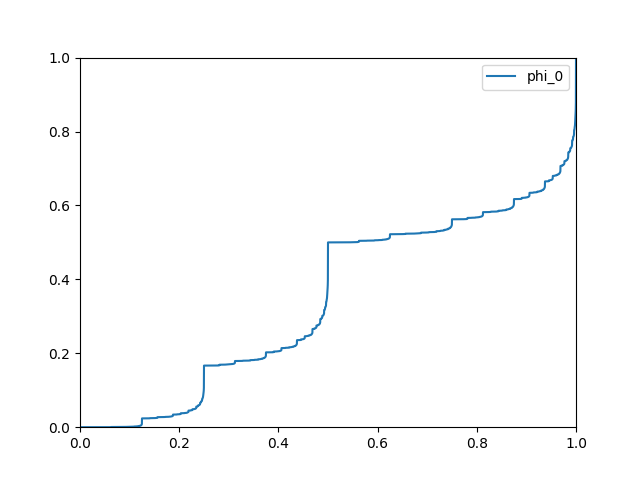}
\captionsetup{font=footnotesize}
\caption{Graph of $\varphi_0$}
\end{minipage}
\begin{minipage}{0.49\columnwidth}
\centering
\includegraphics[scale=0.4]{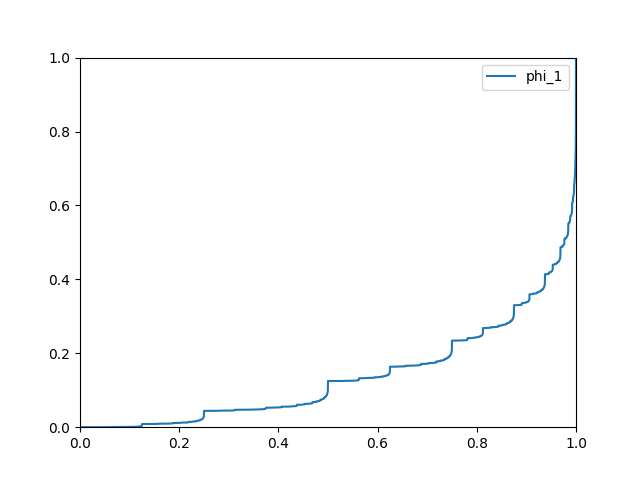}
\captionsetup{font=footnotesize}
\caption{Graph of $\varphi_1$}
\end{minipage}
\end{figure}
\end{Exa}

\cite[Example 5.2]{Okamura2025} gives further examples, but does not provide singularity results. 

\bibliographystyle{plain}
\bibliography{singu-gdR}

\begin{thebibliography}{10}

\bibitem{Barany2018}
Bal{\'a}zs B{\'a}r{\'a}ny, Gergely Kiss, and Istv{\'a}n Kolossv{\'a}ry.
\newblock Pointwise regularity of parameterized affine zipper fractal curves.
\newblock {\em Nonlinearity}, 31(5):1705--1733, 2018.

\bibitem{Berg2000}
L.~Berg and M.~Kr\"uppel.
\newblock De {R}ham's singular function and related functions.
\newblock {\em Z. Anal. Anwendungen}, 19(1):227--237, 2000.

\bibitem{BuescuSerpa2021}
Jorge Buescu and Cristina Serpa.
\newblock Compatibility conditions for systems of iterative functional
  equations with non-trivial contact sets.
\newblock {\em Results Math.}, 76(2):Paper No. 68, 19, 2021.

\bibitem{Girgensohn2006}
Roland Girgensohn, Hans-Heinrich Kairies, and Weinian Zhang.
\newblock Regular and irregular solutions of a system of functional equations.
\newblock {\em Aequationes Math.}, 72(1-2):27--40, 2006.

\bibitem{Hata1985K}
Masayoshi Hata.
\newblock On the functional equation {{\(1/p\cdot \{f(x/p)+\cdots +f((x+p-
  1)/p)\}=\lambda f(\mu x)\)}}.
\newblock {\em J. Math. Kyoto Univ.}, 25:357--364, 1985.

\bibitem{Hata1985}
Masayoshi Hata.
\newblock On the structure of self-similar sets.
\newblock {\em Japan J. Appl. Math.}, 2:381--414, 1985.

\bibitem{Kawaharada2025}
Akane Kawaharada.
\newblock Linear symmetric cellular automata provide {Salem}'s singular
  function.
\newblock {\em Real Anal. Exch.}, 50(1):1--18, 2025.

\bibitem{Kawamura2002}
Kiko Kawamura.
\newblock On the classification of self-similar sets determined by two
  contractions on the plane.
\newblock {\em J. Math. Kyoto Univ.}, 42(2):255--286, 2002.

\bibitem{Matkowski1975}
Janusz Matkowski.
\newblock Integrable solutions of functional equations.
\newblock {\em Diss. Math.}, 127, 1975.

\bibitem{MauldinWilliams1988}
R.~Daniel Mauldin and S.~C. Williams.
\newblock Hausdorff dimension in graph directed constructions.
\newblock {\em Trans. Am. Math. Soc.}, 309(2):811--829, 1988.

\bibitem{MuramotoSekiguchi2020}
Katsushi Muramoto and Takeshi Sekiguchi.
\newblock Directed networks and selfsimilar systems.
\newblock {\em Toyama Math. J.}, 41:1--32, 2020.

\bibitem{Okamura2024}
Kazuki Okamura.
\newblock Quantitative estimates for singularity for conjugate equations driven
  by linear fractional transformations.
\newblock Preprint, {arXiv}:2407.11565.

\bibitem{Okamura2014}
Kazuki Okamura.
\newblock Singularity results for functional equations driven by linear
  fractional transformations.
\newblock {\em J. Theor. Probab.}, 27(4):1316--1328, 2014.

\bibitem{Okamura2019}
Kazuki Okamura.
\newblock Some results for conjugate equations.
\newblock {\em Aequationes Math.}, 93(6):1051--1084, 2019.

\bibitem{Okamura2025}
Kazuki Okamura.
\newblock Construction of graph-directed invariant sets of weak contractions on
  semi-metric spaces.
\newblock {\em Aequationes Math.}, 99(6):2631--2656, 2025.

\bibitem{Rus2001}
Ioan~A. Rus.
\newblock {\em Generalized contractions and applications}.
\newblock Cluj University Press, Cluj-Napoca, 2001.

\bibitem{Serpa2015}
Cristina Serpa and Jorge Buescu.
\newblock Non-uniqueness and exotic solutions of conjugacy equations.
\newblock {\em J. Difference Equ. Appl.}, 21(12):1147--1162, 2015.

\bibitem{Serpa2015N}
Cristina Serpa and Jorge Buescu.
\newblock Piecewise expanding maps and conjugacy equations.
\newblock In {\em Nonlinear maps and their applications. Selected contributions
  from the NOMA 2013 international workshop, Zaragoza, Spain, September 3--4,
  2013}, pages 193--202. Cham: Springer, 2015.

\bibitem{Serpa2017}
Cristina Serpa and Jorge Buescu.
\newblock Constructive solutions for systems of iterative functional equations.
\newblock {\em Constr. Approx.}, 45(2):273--299, 2017.

\bibitem{Tao2011}
Terence Tao.
\newblock {\em An introduction to measure theory}, volume 126 of {\em Graduate
  Studies in Mathematics}.
\newblock American Mathematical Society, Providence, RI, 2011.

\bibitem{Zdun2001}
Marek~Cezary Zdun.
\newblock On conjugacy of some systems of functions.
\newblock {\em Aequationes Math.}, 61(3):239--254, 2001.

\end{thebibliography}

\end{document}